\documentclass[letterpaper,10pt]{amsart}
\usepackage[intlimits]{gafmath} 
\usepackage{amssymb,amsfonts} 
\usepackage{graphicx}
\usepackage{subfigure}
\usepackage[alphabetic]{amsrefs}
\usepackage[pagebackref]{hyperref}

\hypersetup{pdfauthor={G. Austin Ford},pdftitle={The fundamental solution and Strichartz estimates for the Schr\"odinger equation on flat euclidean cones}}

\numberwithin{theorem}{section}

\DeclareMathOperator{\erfc}{erfc}

\begin{document}

\title[Schr\"odinger equation on flat euclidean cones]{The fundamental solution and Strichartz estimates for the Schr\"odinger equation on flat euclidean cones}

\author{G.\ Austin Ford}
\address{Department of Mathematics, Northwestern University}
\email{aford@math.northwestern.edu}

\begin{abstract}
We study the Schr\"odinger equation on a flat euclidean cone $\mathbb{R}_+ \times \mathbb{S}^1_\rho$ of cross-sectional radius $\rho > 0$, developing asymptotics for the fundamental solution both in the regime near the cone point and at radial infinity.  These asymptotic expansions remain uniform while approaching the intersection of the ``geometric front,'' the part of the solution coming from formal application of the method of images, and the ``diffractive front'' emerging from the cone tip.  As an application, we prove Strichartz estimates for the Schr\"odinger propagator on this class of cones.
\end{abstract}

\maketitle


\setcounter{section}{-1}

\section{Introduction}\label{sect:Intro}  
In this paper, we study the initial value problem for the Schr\"odinger equation,
\begin{equation}\label{eq:Schro IVP}
\left\{
\begin{split}
\left(D_t - \Delta\right) u(t,r,\theta) &= 0 \\
u(0,r,\theta) &= u_0(r,\theta) \; ,
\end{split}
\right.
\end{equation}
on a flat euclidean cone.  This is an incomplete manifold $C(\mathbb{S}^1_\rho) \defeq \mathbb{R}_+ \times \mathbb{S}^1_\rho$, where $\mathbb{S}^1_\rho \defeq \mathbb{R} \big/ 2\pi\rho\mathbb{Z}$ is the circle of radius $\rho > 0$, equipped with the metric $g(r,\theta) = dr^2 + r^2 d\theta^2$, and the Laplacian $\Delta$ is taken to be the Friedrichs extension of $\Delta_g \big\vert_{\mathcal{C}^\infty_c\!\left(C(\mathbb{S}^1_\rho)\right)}$.  Specifically, we are interested in the behavior of the fundamental solution $e^{it\Delta} \delta_{(r_0,\theta_0)}$ to \eqref{eq:Schro IVP}.

We begin by using Cheeger's functional calculus for cones, developed first in \cite{Che}, to show that the Schr\"odinger propagator on $C(\mathbb{S}^1_\rho)$ has the series representation
\begin{multline}\label{eq:intro std cone SSK}
K_{e^{it\Delta}}(r_1,\theta_1,r_2,\theta_2) = \\
-\frac{\exp\!\left[\frac{r_1^2 + r_2^2}{4it}\right]}{4\pi i\rho t}  \left\{ J_0\left( \frac{r_1 r_2}{2t} \right) + 2\sum_{j = 1}^\infty i^{j/\rho} \, J_{j/\rho}\!\left( \frac{r_1 r_2}{2t} \right) \cos\!\left[ \frac{j}{\rho} \left( \theta_1 - \theta_2 \right) \right] \right\} \; ,
\end{multline}
where $J_\nu(x)$ is the Bessel function of order $\nu$.  Employing an integral representation for $J_\nu(x)$ due to Schl\"afli, we then show that the quantity in braces in \eqref{eq:intro std cone SSK} can be represented by the loop integral
\begin{multline}\label{eq:intro S(x,eta) loop}
S(x,\eta) = \frac{1}{4\pi} \int_C \exp\!\left[ \frac{x}{2} \left(v - \frac{1}{v} \right) \right] \left\{ \cot \left[ \frac{\frac{\pi}{2} + \eta + i \log[v]}{2\rho} \right] \right. \\
\left. \mbox{} + \cot \left[ \frac{\frac{\pi}{2} - \eta + i \log[v]}{2\rho} \right] \right\} \frac{dv}{v} \; ,
\end{multline}
where $x \defeq \frac{r_1 r_2}{2t}$ and $\eta \defeq \theta_1 - \theta_2$.

Due to the fact that the amplitude of $S(x,\eta)$ has poles which move with $\eta$ and collide with the stationary points $v = \pm i$ of the phase $\frac{1}{2} \left( v - \frac{1}{v} \right)$, the standard techniques of asymptotic analysis will not produce an expansion of \eqref{eq:intro S(x,eta) loop} which is uniform in $\eta$.  However, by modifying a version of the method of steepest descent due to van der Waerden \cite{vdW}, we are able to produce a uniform asymptotic expansion of $S(x,\eta)$ in decreasing powers of $x$ as $x \To \infty$, and this leads us to an expansion for $K_{e^{it\Delta}}$ in decreasing powers of $\frac{r_1 r_2}{2t}$ as this quantity approaches infinity.  Namely, we show
\begin{multline}\label{eq:intro asymptotics away from interface}
K_{e^{it\Delta}}(r_1,\theta_1,r_2,\theta_2) \\
\mbox{} \sim \frac{1}{t} \sum_{\alpha = \pm 1} \left\{ \mathsf{G}_0^\alpha + \mathsf{G}^\alpha_{-\frac{1}{2}} \left(\frac{r_1 r_2}{2t} \right)^{-\frac{1}{2}} + \sum_{k=0}^\infty \mathsf{D}^\alpha_{-\frac{2k+1}{2}} \left(\frac{r_1 r_2}{2t}\right)^{-\frac{2k + 1}{2}} \right\} \qquad \text{as $\dfrac{r_1 r_2}{2t} \To \infty$}
\end{multline}
in the regime away from poles of \eqref{eq:intro S(x,eta) loop} coinciding with the stationary points of the phase, i.e.\ $\theta_1 - \theta_2 \not\equiv -\pi$, $0$, or $\pi \mod{2\pi\rho}$.  These functions $\mathsf{G}^\alpha_s$ and $\mathsf{D}^\alpha_s$ are piecewise-smooth and bounded in all variables, and their precise definitions can be found in Section \ref{sect:Asymptotics}.  We also provide asymptotics as $\frac{r_1 r_2}{2t} \To 0$, showing
\begin{equation}
K_{e^{it\Delta}}(r_1,\theta_1,r_2,\theta_2) = -\frac{1}{4\pi i \rho t} \exp\!\left[ \frac{r_1^2 + r_2^2}{4it} \right] \left(1 + \mathrm{O}\left( \left(\frac{r_1 r_2}{2t} \right)^\sigma \right)\right)
\end{equation}
in this regime, where $\sigma \defeq \min\!\left(2,\frac{1}{\rho}\right)$.  This only makes use of elementary estimates for Bessel functions.

This asymptotic expansion \eqref{eq:intro asymptotics away from interface} shows that the Schr\"odinger kernel $K_{e^{it\Delta}}$ is separated into two parts, the ``geometric'' factors $\left(\frac{r_1 r_2}{2t}\right)^{-s} \mathsf{G}_s^\alpha$ analogous to those that would arise from the formal application of the method of images, and the $\left(\frac{r_1 r_2}{2t}\right)^{-s} \mathsf{D}^\alpha_s$ terms arising from a ``diffractive effect'' emerging from the cone tip.  The diffractive terms have noticeably better decay vis-\`a-vis $\frac{r_1 r_2}{2t}$, being of order $-\frac{1}{2}$ in this variable, whereas the geometric terms as a whole are of order $0$.  This is analogous to the classical results for the wave equation of Sommerfeld \cite{Som} and Friedlander \cite{Fri} in the presence of obstacles and later work of Cheeger and Taylor in the setting of product cones \cite{CheTay1} \cite{CheTay2} and of Melrose and Wunsch for manifolds with cone points \cite{MelWun}.  In each case, they show the diffractive front is $\frac{1}{2}$ degree ``weaker,'' i.e.\ more regular in an appropriate sense.  It is also morally consistent with the parametrix construction of Hassell and Wunsch \cite{HasWun} for the Schr\"odinger equation on scattering manifolds, where they show the leading order part of the propagator is given by the sojourn relation.

As an application of our asymptotic expansion, we prove the Strichartz estimates 
\begin{equation}\label{eq:intro Strichartz 1}
\left\| \mathcal{U}(t) f(r,\theta) \right\|_{L^p_t L^q(r \, dr d\theta)} \lesssim \left\|f\right\|_{L^2(r \, dr d\theta)}
\end{equation}
\begin{equation}\label{eq:intro Strichartz 2}
\left\| \int \mathcal{U}(-s) F(s,r,\theta) \, ds \right\|_{L^2(r \, dr d\theta)} \lesssim \left\| F(t,r,\theta) \right\|_{L^{p'}_t L^{q'}(r \, dr d\theta)}
\end{equation}
\begin{equation}\label{eq:intro Strichartz 3}
\left\| \int_{s < t} \mathcal{U}(t-s) F(s,r,\theta) \, ds \right\|_{L^p_t L^q(r \, dr d\theta)} \lesssim \left\| F(t,r,\theta) \right\|_{L^{\tilde{p}}_t L^{\tilde{q}}(r \, dr d\theta)}
\end{equation}
for the Schr\"odinger propagator $\mathcal{U}(t) \defeq e^{it\Delta}$ using the theorem of Keel and Tao \cite{KeeTao}.  Planchon and Stalker show these estimates for rational $\rho$ in their manuscript \cite{PlaSta}, though their method does not seem to generalize to irrational cross-sectional radius.  We also note there has also been related work in the case of exterior domains \cite{BurGerTzv} \cite{Iva} and in the presence of inverse square potentials \cite{BurPlaStaTah}.

It is worth remarking that Deser and Jackiw produce analogous expressions to \eqref{eq:intro std cone SSK} and \eqref{eq:intro S(x,eta) loop} in \cite{DesJac}, though their integral representation is not as well suited for our purposes as the one provided here.

The structure of the paper is as follows.  In Section \ref{sect:Cheeger fl calc}, we review Cheeger's functional calculus for flat cones.  In Section \ref{sect:SchroEqn}, we specialize the setting to a cone over the circle $\mathbb{S}^1_\rho$ and determine the Schr\"odinger kernel as a Fourier series with Bessel function coefficients.  Section \ref{sect:FundSoln} is dedicated to the construction of the integral representation \eqref{eq:intro std cone SSK} for the Schr\"odinger kernel.  We then develop its asymptotics in Section \ref{sect:Asymptotics}, utilizing van der Waerden's method of steepest descent to counteract the difficulties found approaching the interface of the geometric and diffractive fronts.  Finally, in Section \ref{sect:Strichartz} we use the information gained from the development of the asymptotic expansion to prove Strichartz estimates for the Schr\"odinger propagator.

\subsection*{Acknowledgements}  The author would like to thank Jared Wunsch for suggesting the problem and for encouragement and helpful conversations.  He also thanks Fabrice Planchon for access to the unpublished manuscript \cite{PlaSta} and an anonymous referee for comments improving the exposition.

\section{Cheeger's functional calculus for cones}\label{sect:Cheeger fl calc}

We shall begin by establishing some notation and briefly recalling Cheeger's functional calculus for flat cones; for a thorough discussion of these results and other applications, we refer to Cheeger's article with Taylor \cite{CheTay1} or to the second book of Taylor's treatise \cite{Tay2}.

Let $M^m$ be a closed $\mathcal{C}^\infty$ manifold with Riemannian metric $h$, and let $C(M) \defeq \mathbb{R}_+ \times M$ be the cone over $M$.  We give $C(M)$ the Riemannian metric
\begin{equation}
g(r,x) = dr^2 + r^2 \, h(x) \; .
\end{equation}
The positive Laplacian on $C(M)$ then takes the form
\begin{equation}
\Delta_g = - \del_r^2 - \frac{m}{r} \, \del_r + \frac{1}{r^2} \, \Delta_h \; ,
\end{equation}
where $\Delta_h$ is the (positive) Laplacian on the cross-sectional manifold $M$.  Writing $\left\{ \mu_j \right\}_{j=0}^\infty$ for the eigenvalues of $\Delta_h$ (with multiplicity) and $\left\{ \varphi_j : M \To \mathbb{C} \right\}_{j=0}^\infty$ for the corresponding eigenfunctions, we define the rescaled eigenvalues $\nu_j$ by
\begin{equation}
\nu_j \defeq \left( \mu_j + \frac{(m-1)^2}{4} \right)^{\frac{1}{2}} \; .
\end{equation}

Henceforth, we take $\Delta_g$ to be the Friedrichs extension of the above Laplace operator on functions.  As is well known, a suitable function $f : C(M) \To \mathbb{C}$ gives rise to an operator $f(\Delta_g)$ via spectral theory.  Cheeger's separation of variables approach shows that the Schwartz kernel of $f(\Delta_g)$, which we will write as $K_{f(\Delta_g)}$, takes the form
\begin{equation}\label{eq:prop SK}
K_{f(\Delta_g)}(r_1,x_1,r_2,x_2) = \left(r_1 r_2\right)^{-\frac{m-1}{2}}\sum_{j=0}^\infty \tilde{K}_{f(\Delta_g)}(r_1,r_2,\nu_j) \, \varphi_j(x_1) \, \overline{\varphi_j(x_2)} \; ,
\end{equation}
where the radial coefficient $\tilde{K}_{f(\Delta)}$ is
\begin{equation}\label{eq:SK radial part}
\tilde{K}_{f(\Delta_g)}(r_1,r_2,\nu) \defeq \int_{0}^\infty f(\lambda^2) \, J_\nu(\lambda r_1) \, J_\nu(\lambda r_2) \, \lambda \, d\lambda \; .
\end{equation}
Here, $J_\nu(x)$ is the Bessel function of order $\nu$,
\begin{equation}
J_\nu(x) \defeq \sum_{j=0}^\infty \frac{(-1)^j}{j! \, \Gamma(\nu+j+1)} \left( \frac{x}{2} \right)^{\nu + 2j} \; .
\end{equation}

\section{The Schr\"odinger equation on flat cones}\label{sect:SchroEqn}

We now specialize to the case where the cross-sectional manifold $M$ is the circle of radius $\rho > 0$, which we will write as $\mathbb{S}^1_\rho \defeq \mathbb{R} \big/ 2 \pi \rho \mathbb{Z}$.  Equipping it with the metric $h(\theta) = d\theta^2$ inherited from $\mathbb{R}$, the Laplace operator on $\mathbb{S}^1_\rho$ is $\Delta_h = - \del_\theta^2$, and its eigenvalues and eigenfunctions are
\begin{equation}
\mu_j = \frac{j}{\rho}, \qquad \varphi_j^{\pm}(\theta) = \frac{1}{\sqrt{2\pi \rho}} \, \exp\!\left[ \pm \frac{ij\theta}{\rho} \right], \qquad \text{where $j = 0,1,2,\dots$.}
\end{equation}
Note that the positive eigenvalues $\mu_j > 0$ have multiplicity 2, whereas $\mu_0 = 0$ has multiplicity 1.  Moving to the cone $C(\mathbb{S}^1_\rho)$ with metric $g = dr^2 + r^2 \, d\theta^2$, the associated Laplacian $\Delta_g$ is
\begin{equation}
\Delta_g = - \del_r^2 - \frac{1}{r} \, \del_r - \frac{1}{r^2} \, \del_\theta^2 \; ,
\end{equation}
which we see to be the standard Laplacian for $\mathbb{R}^2$ written in polar coordinates.  In the following, we will write $\Delta$ for the Friedrichs extension of this Laplace operator on $C(\mathbb{S}^1_\rho)$, understanding that the metric dependence is implicit.

Consider the solution operator for the Schr\"odinger equation \eqref{eq:Schro IVP},
\begin{equation}
\mathcal{U}(t) = e^{i t \Delta} \; .
\end{equation}
Using Cheeger's formulae \eqref{eq:prop SK} and \eqref{eq:SK radial part} for the Schwartz kernel of functions of the Laplacian, we see that
\begin{equation}\label{eq:cone SSK 1}
K_{e^{it\Delta}}(r_1,\theta_1,r_2,\theta_2) = \frac{1}{2\pi\rho}\sum_{j=-\infty}^\infty \tilde{K}_{e^{it\Delta}}\!\left(r_1,r_2,\frac{|j|}{\rho}\right) \, \exp\!\left[ \frac{ij}{\rho} \left(\theta_1 - \theta_2\right) \right] \; ,
\end{equation}
where
\begin{equation}
\tilde{K}_{e^{it\Delta}}(r_1,r_2,\nu) = \int_0^\infty e^{it\lambda^2} \, J_\nu(\lambda r_1) \, J_\nu(\lambda r_2) \, \lambda \, d\lambda \; .
\end{equation}
Letting $t = i s$, we obtain an expression for the heat kernel $e^{-s\Delta}$.  In particular, applying Weber's second exponential integral \cite{Wat}*{$\S 13.31(1)$} to the radial coefficient $\tilde{K}_{e^{-s\Delta}}$ gives us the expression
\begin{equation}
\tilde{K}_{e^{-s\Delta}}(r_1,r_2,\nu) = \frac{\exp\!\left[-\frac{r_1^2 + r_2^2}{4s}\right]}{2 s} \, I_\nu\!\left(\frac{r_1r_2}{2s}\right) \; ,
\end{equation}
where $I_\nu(x)$ is the modified Bessel function of the first kind,
\begin{equation}
I_\nu(x) \defeq \sum_{j=0}^\infty \frac{1}{j! \, \Gamma(\nu+j+1)} \left( \frac{x}{2} \right)^{\nu + 2j} \; .
\end{equation}
Analytic continuation in $s$ and setting $s = -i t$ then returns an expression for the Schr\"odinger propagator,
\begin{equation}
\tilde{K}_{e^{it\Delta}}(r_1,r_2,\nu) = \frac{i^{\nu+1} \, \exp\!\left[\frac{r_1^2 + r_2^2}{4it}\right]}{2t} \, J_\nu\!\left(\frac{r_1r_2}{2t}\right) \; ,
\end{equation}
where we use the fact that $I_\nu(ix) = i^\nu \, J_\nu(x)$.  Substituting this into \eqref{eq:cone SSK 1} and combining the exponentials $\exp\!\left[\frac{ij}{\rho} \left(\theta_1 - \theta_2\right) \right]$ and $\exp\!\left[-\frac{ij}{\rho} \left(\theta_1 - \theta_2\right) \right]$ for positive $j$, we obtain the following proposition.

\begin{proposition}
The Schr\"odinger propagator $e^{i t \Delta}$ on $C(\mathbb{S}^1_\rho)$ has Schwartz kernel
\begin{multline}\label{eq:std cone SSK}
K_{e^{it\Delta}}(r_1,\theta_1,r_2,\theta_2) = \\
-\frac{\exp\!\left[\frac{r_1^2 + r_2^2}{4it}\right]}{4\pi i\rho t}  \left\{ J_0\left( \frac{r_1 r_2}{2t} \right) + 2\sum_{j = 1}^\infty i^{j/\rho} \, J_{j/\rho}\!\left( \frac{r_1 r_2}{2t} \right) \cos\!\left[ \frac{j}{\rho} \left( \theta_1 - \theta_2 \right) \right] \right\} \; .  \qed
\end{multline}
\end{proposition}

\section{An integral representation for the fundamental solution}\label{sect:FundSoln}

The next step in our analysis is to transform the expression \eqref{eq:std cone SSK} for the Schwartz kernel of the propagator into one more amenable to calculation.  Before we start, we simplify the calculation by introducing the dummy variables $x$ and $\eta$, defined to be
\begin{equation}
x \defeq \frac{r_1 r_2}{2t} \qquad \text{and} \qquad \eta \defeq \theta_1 - \theta_2 \; ;
\end{equation}
these are the arguments of the Bessel functions and the cosines respectively.  We also introduce the name $S(x,\eta)$ for the quantity in braces in \eqref{eq:std cone SSK}, i.e.\
\begin{equation}\label{eq:S(x,eta) defn}
S(x,\eta) \defeq J_0(x) + 2\sum_{j = 1}^\infty i^{j/\rho} \, J_{j/\rho}(x) \cos\!\left[ \frac{j}{\rho} \, \eta \right] \; .
\end{equation}
This function $S(x,\eta)$ will be the primary target for our analysis, and its asymptotics will provide asymptotics of $K_{e^{it\Delta}}$.

\begin{lemma}
The function $S(x,\eta)$ has a loop integral representation
\begin{multline}\label{eq:S(x,eta) loop}
S(x,\eta) = \frac{1}{4\pi} \int_C \exp\!\left[ \frac{x}{2} \left(v - \frac{1}{v} \right) \right] \left\{ \cot \left[ \frac{\frac{\pi}{2} + \eta + i \log[v]}{2\rho} \right] \right. \\
\left. \mbox{} + \cot \left[ \frac{\frac{\pi}{2} - \eta + i \log[v]}{2\rho} \right] \right\} \frac{dv}{v} \; ,
\end{multline}
where $C$ is a contour starting at $-\infty$, encircling the unit circle in a counterclockwise direction, and returning to $-\infty$.
\end{lemma}

\begin{proof}
Consider the Schl\"afli loop integral representation\footnote{The notation ``$\int_{-\infty}^{(0+)}$'' signifies that the contour begins at $-\infty$, wraps around the origin with positive (counterclockwise) orientation, and returns to $-\infty$.} for the Bessel function $J_\nu(x)$ \cite{Wat}*{\S6.2(2)},
\begin{equation}
J_\nu(x) = \frac{1}{2\pi i} \int_{-\infty}^{(0+)} \exp\!\left[ \frac{x}{2} \left(v - \frac{1}{v}\right) \right] \frac{dv}{v^{\nu + 1}} \; .
\end{equation}
Substituting this formula for the Bessel functions in the definition of $S(x,\eta)$ \eqref{eq:S(x,eta) defn} and exchanging the summation and integration, we have the expression
\begin{equation}\label{eq:S(x,eta) loop 1}
S(x,\eta) = \frac{1}{2\pi i} \int_{-\infty}^{(0+)} \exp\!\left[ \frac{x}{2} \left(v - \frac{1}{v}\right) \right] \left\{ 1 + 2 \sum_{j=1}^\infty \frac{i^{j/\rho} \, \cos\!\left[\frac{j}{\rho} \, \eta \right]}{v^{j/\rho}} \right\} \frac{dv}{v} \; .
\end{equation}
This exchange is justifiable by taking the contour to be sufficiently far away from the origin; choosing a contour so that $|v| > 1 + \varepsilon$ for some $\varepsilon > 0$ will ensure the resulting integral is absolutely convergent.

Under these same conditions, we can take advantage of the fact that the quantity in braces in \eqref{eq:S(x,eta) loop 1} is a sum of two geometric series:
\begin{equation}\label{eq:S(x,eta) loop 2}
\begin{split}
S(x,\eta) &= \frac{1}{2\pi i} \int_{-\infty}^{(0+)} \exp\!\left[ \frac{x}{2} \left(v - \frac{1}{v}\right) \right] \left\{ 1 + \sum_{j=1}^\infty \exp\!\left[ \frac{ij \left( \frac{\pi}{2} + \eta + i \log[v] \right)}{\rho} \right] \right. \\
&\hspace*{15em} \left. \mbox{} + \sum_{j=1}^\infty \exp\!\left[ \frac{ij \left( \frac{\pi}{2} - \eta + i \log[v] \right)}{\rho} \right] \right\} \frac{dv}{v} \\
&= \frac{1}{2\pi i} \int_{-\infty}^{(0+)} \exp\!\left[ \frac{x}{2} \left(v - \frac{1}{v}\right) \right] \left\{ 1 + \frac{\exp\!\left[\frac{i}{\rho} \left( \frac{\pi}{2} + \eta + i \log[v] \right) \right]}{1 - \exp\!\left[\frac{i}{\rho} \left( \frac{\pi}{2} + \eta + i \log[v] \right)\right]} \right. \\
&\hspace*{15em} \left. \mbox{} + \frac{\exp\!\left[\frac{i}{\rho} \left( \frac{\pi}{2} - \eta + i \log[v] \right) \right]}{1 - \exp\!\left[\frac{i}{\rho} \left( \frac{\pi}{2} - \eta + i \log[v] \right)\right]}\right\} \frac{dv}{v} 
\end{split}
\end{equation}
Here, the logarithm is chosen to have its branch along the nonpositive real axis so as not to interfere with the integration contour.  Now, we note the equality
\begin{equation}
\frac{1}{2} + \frac{\exp[i\alpha]}{1 - \exp[i\alpha]} = \frac{i}{2} \, \cot\!\left[ \frac{\alpha}{2} \right] \; .
\end{equation}
Substituting this into the above gives us the desired form.
\end{proof}

\begin{remark}
Deser and Jackiw use a similar method in \cite{DesJac}, though they apply it to another of Schl\"afli's integral representations \cite{Wat}*{\S6.2(3)}.  Our choice has the merit of producing a simpler expression for the exponential phase in $S(x,\eta)$, facilitating its asymptotic development in what follows.  $\diamond$
\end{remark}

Meromorphic continuation in $v$ of the integrand in \eqref{eq:S(x,eta) loop} shows that it is holomorphic away from the logarithmic branch along the nonpositive real axis and a finite number of poles.  These poles all lie on the unit circle and are of the form $e^{i\varphi}$ for $\varphi$ in the set $\mathcal{P}_\rho(\pm \eta)$ of ``pole phases,''
\begin{equation}
\mathcal{P}_\rho(\pm \eta) \defeq \left\{ \frac{\pi}{2} \pm \eta + 2 \pi \rho k ;\; k \in \mathbb{Z} \right\} \cap [-\pi,\pi) \; .
\end{equation}
The sign of $\eta$ here denotes to which summand of the amplitude the pole belongs, and the intersection with $[-\pi,\pi)$ restricts the poles to lying on a single sheet of the universal cover of the punctured plane.  This observation allows us to deform our contour $C$ as we wish.  

\section{Asymptotics of the fundamental solution}\label{sect:Asymptotics}

We will now calculate the asymptotics of $S(x,\eta)$ as $x \To 0$ and  $x \To \infty$ for general cross-sectional radius $\rho$.  These will in turn give us the asymptotics of the fundamental solution $e^{it\Delta} \delta_{(r_0,\theta_0)}$ of \eqref{eq:Schro IVP} which we will use to prove Strichartz estimates in Section \ref{sect:Strichartz}.

We begin by addressing the $x \To 0$ regime in the following proposition.
\begin{proposition}\label{thm:propagator asymptotics near cone tip}
The Schwartz kernel of $e^{it\Delta}$ has leading order asymptotics
\begin{multline}
K_{e^{it\Delta}}(r_1,\theta_1,r_2,\theta_2) \\
= -\frac{1}{4\pi i\rho t} \exp\!\left[ \frac{r_1^2 + r_2^2}{4it} \right] \left(1 + \mathrm{O}\!\left(\left(\frac{r_1 r_2}{2t}\right)^\sigma\right)\right) \qquad \text{as $\dfrac{r_1 r_2}{2t} \To 0$} \; ,
\end{multline}
where $\sigma \defeq \min\!\left(2,\frac{1}{\rho}\right)$.
\end{proposition}

\begin{proof}
We begin with the bound \cite{Wat}*{$\S$3.31(1)}
\begin{equation}
\left| J_\nu(x) \right| \leqslant \frac{1}{\Gamma(\nu+1)} \left(\frac{x}{2}\right)^\nu \; ,
\end{equation}
valid for Bessel functions with $x$ real and $\nu > -\frac{1}{2}$.  For $|x| < 2$, this estimate implies
\begin{equation}
\begin{aligned}
\left|S(x,\eta) - 1\right| &\defeq \left| \left(J_0(x) - 1\right) + 2 \sum_{j=1}^\infty i^{j/\rho} \, J_{j/\rho}(x) \, \cos\!\left[\frac{j}{\rho} \, \eta\right] \right| \\
&\leqslant \left|\sum_{j=1}^\infty \frac{(-1)^j}{\left(j!\right)^2} \left(\frac{x}{2}\right)^{2j} \right| + 2 \sum_{j=1}^\infty \frac{1}{\Gamma\!\left(\frac{j}{\rho} + 1\right)} \left(\frac{x}{2}\right)^{j/\rho} \\
&\leqslant \sum_{j=1}^\infty \left(\frac{x^2}{4}\right)^j + \frac{2}{\Gamma\!\left(\frac{1}{\rho} + 1\right)} \sum_{j=1}^\infty \left(\frac{x}{2}\right)^{j/\rho} \\
&= \frac{x^2}{4 - x^2} + \frac{2}{\Gamma\!\left(\frac{1}{\rho} + 1\right)} \frac{x^{1/\rho}}{2^{1/\rho} - x^{1/\rho}} \; .
\end{aligned}
\end{equation}
This proves the desired asymptotics.
\end{proof}

To handle the $x \To \infty$ regime, we proceed by applying a modified version of the method of steepest descent\footnote{For the details of the standard method of steepest descent, see Olver's book \cite{Olv} or any book on asymptotics or special functions.} developed in van der Waerden's article \cite{vdW}.  Our approach will differ from van der Waerden's in that our poles move with changing $\eta$.  This produces a spurious singularity if we follow \cite{vdW} to the letter, however a straightforward modification prevents this kind of degeneration.

Before diving into the calculation, we define
\begin{equation}\label{eq:S_alpha(x,eta) defn}
S_\alpha(x,\eta) \defeq \frac{1}{4\pi} \int_C \exp\!\left[ \frac{x}{2} \left(v - \frac{1}{v} \right) \right] \cot \left[ \frac{\frac{\pi}{2} + \alpha\eta + i \log[v]}{2\rho} \right] \frac{dv}{v} \; ,
\end{equation}
where $\alpha = \pm 1$.  Thus $S(x,\eta) = S_{+}(x,\eta) + S_{-}(x,\eta)$.

\subsection{Van der Waerden's change of variables}\label{sect:vdW change of vars}

We introduce into the integral \eqref{eq:S_alpha(x,eta) defn} the change of variables
\begin{equation}\label{eq:v->u change of vars}
u(v) = \frac{1}{2} \left( v - \frac{1}{v} \right) \; ,
\end{equation}
taking the phase of $S_\alpha(x,\eta)$ as our new base variable.  This map $u : \mathbb{C} \To \mathbb{C}$ is a branched double cover of the complex plane with branch points at $i$ and $-i$.  The two sheets of this cover are the images of the reverse change of variables maps
\begin{equation}\label{eq:u->v change of vars}
v_\pm(u) = u \pm \left(u^2 + 1\right)^\frac{1}{2} \; ,
\end{equation}
namely
\begin{equation}
\begin{aligned}
v_-(\mathbb{C}) &= \left\{ v \in \mathbb{C} ;\; \Re[v] < 0 \right\} \cup \left\{ ib \in i\mathbb{R} ;\; |b| \leqslant 1 \right\}  \\
v_+(\mathbb{C}) &= \left\{ v \in \mathbb{C} ;\; \Re[v] > 0 \right\} \cup\left\{ ib \in i\mathbb{R} ;\; |b| \geqslant 1 \right\}
\end{aligned}
\end{equation}
as shown in Figure \ref{fig:v-plane}.  Here we take the principal branch of the square root, requiring $\Re\!\left[z^\frac{1}{2}\right] \geqslant 0$.  Our original variable $v$ is therefore a multi-valued function of $u$ whose branches are given by $v_\pm$.

Since one part of $C$ lies on the $v_-$-sheet and the other on the $v_+$-sheet, the image contour $u(C)$ crosses the branch cuts emanating from $u = \pm i$; see Figure \ref{fig:u-plane} for an illustration.  We shall write $C_\pm$ for the part of the contour $C$ lying in the sheet $v_\pm(\mathbb{C})$.
\begin{figure}
	\subfigure[The $v$-plane and its decomposition into the $v_\pm$-sheets]{\includegraphics{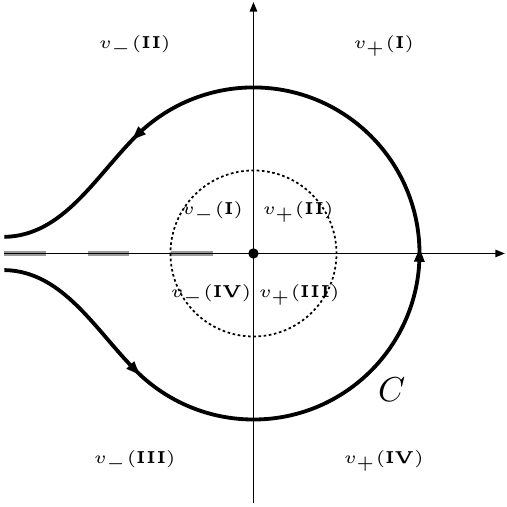}\label{fig:v-plane}}
	\subfigure[The $u$-plane]{\includegraphics{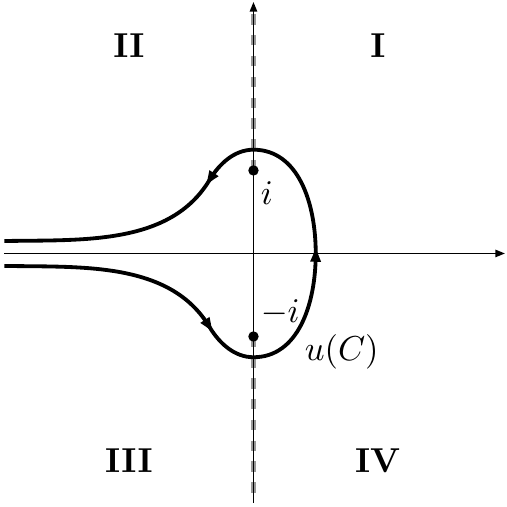}\label{fig:u-plane}}
\caption{The change of variables $v \leadsto u$ and its mapping properties (quadrants of the plane denoted by bold roman numerals, unit circle by thin dashed line, logarithmic cut by thick gray dashes, and branch cuts by thin gray dashes)}
\label{fig:vdW change of vars}
\end{figure}
Expressing $S_\alpha(x,\eta)$ in terms of $u$ produces the equation
\begin{equation}\label{eq:S_alpha(x,eta) via u}
S_\alpha(x,\eta) = \frac{1}{4\pi} \int_{u(C)} e^{x u} \, \cot\!\left[ \frac{\frac{\pi}{2} + \alpha \eta + i \log\!\left[ u + \sigma \left(u^2 + 1\right)^\frac{1}{2} \right]}{2\rho} \right] \frac{du}{\sigma \left(u^2 + 1\right)^\frac{1}{2}} \; ,
\end{equation}
where $\sigma = \sigma(u)$ is $\pm 1$ on $u(C_\pm)$ and serves to correct the branch of the square root.  We remark that the poles of the integrand, located initially at $e^{i\varphi}$ in the $v$-plane for $\varphi$ in $\mathcal{P}_\rho(\alpha\eta)$, move to the points $u\!\left(e^{i\varphi}\right) = i \sin(\varphi)$ in the segment of the imaginary axis between $-i$ and $i$ in the $u$-plane.

\subsection{Contour deformations and local uniformizations}\label{sect:contour deformations}
For the remainder of this section, we will assume that none of the poles $u = i \sin(\varphi)$ coincide with the branch points at $u = \pm i$, i.e.\
\begin{equation}
\alpha\eta \not\equiv -\pi \ \text{or} \ 0 \mod{2\pi\rho} \; .
\end{equation}
As we shall see, this assumption puts us in the regime where the ``geometric front,'' which is the part of the fundamental solution that arises from formal application of the method of images, and the ``diffractive front'' emanating from the cone tip do not interact.

We now deform the contour $u(C)$.  This starts by separating $u(C)$ into $u(C_-)$ and $u(C_+)$ (see Figure \ref{fig:separated u-plane contours}), the parts of $u(C)$ on which the integrand is single-valued; in what follows, we always ensure that after the deformations the endpoints of these contours match.
\begin{figure}
	\subfigure[The original $u(C_-)$]{\includegraphics{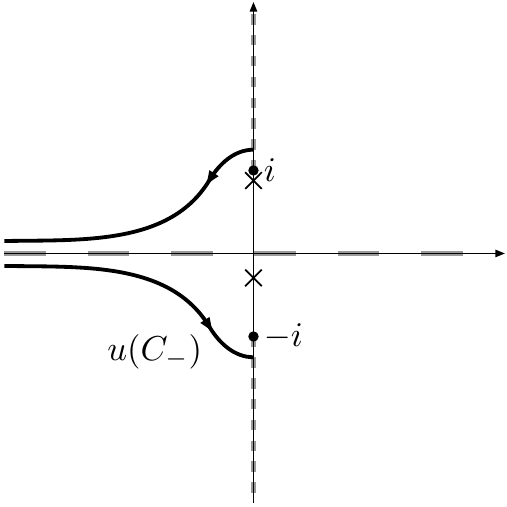}}
	\subfigure[The original $u(C_+)$]{\includegraphics{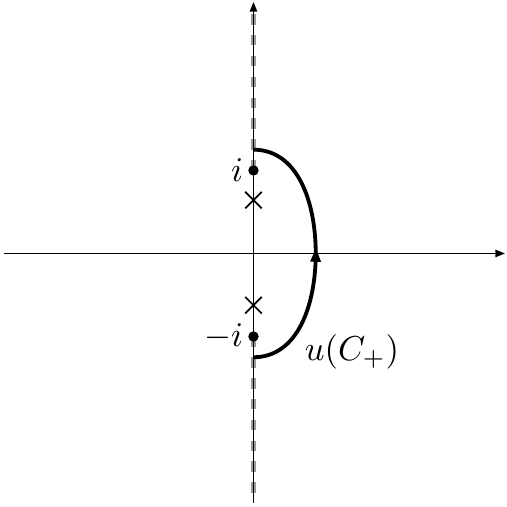}}
\caption{The contour $u(C)$ separated into $u(C_-)$ and $u(C_+)$ and example poles of the integrand ($\rho = \frac{1}{\sqrt{21}}$, $\eta = \frac{2\pi}{9}$)}
\label{fig:separated u-plane contours}
\end{figure}
Changing $u(C_-)$ or $u(C_+)$ will vary the contributions of the individual pieces to the contour integral, but the integral over the entire contour $u(C) = u(C_-) \cup u(C_+)$ will remain the same.  

We replace the $u(C_-)$ contour with one consisting of straight horizontal lines running from (negative) infinity to the branch points, shown in Figure \ref{fig:deformed uCminus}, and we denote by $C_-^\pm$ the piece of $u(C_-)$ containing $u=\pm i$.
\begin{figure}
	\subfigure[The deformed $u(C_-)$]{\includegraphics{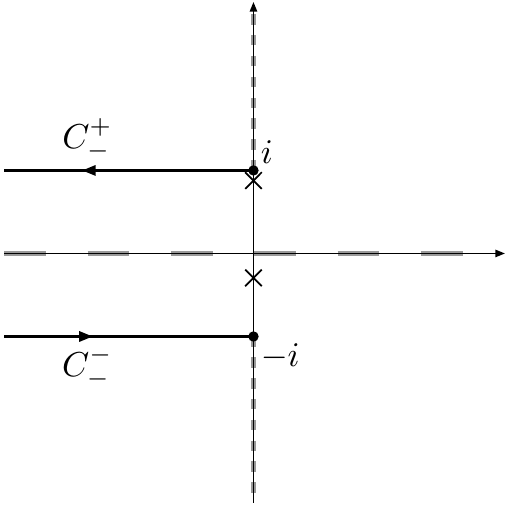}\label{fig:deformed uCminus}}
	\subfigure[The transitional $u(C_+)$]{\includegraphics{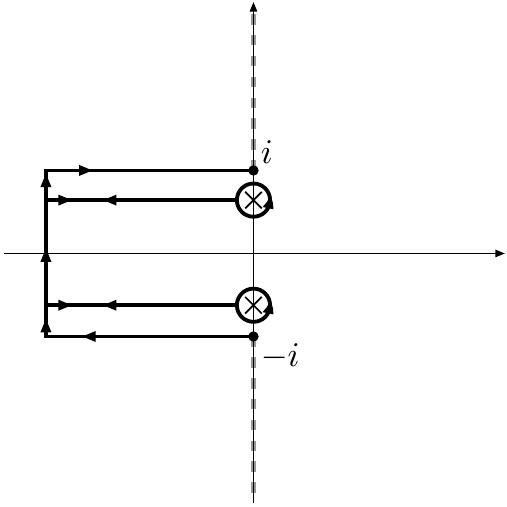}\label{fig:trans uCplus}}
	\subfigure[The deformed $u(C_+)$]{\includegraphics{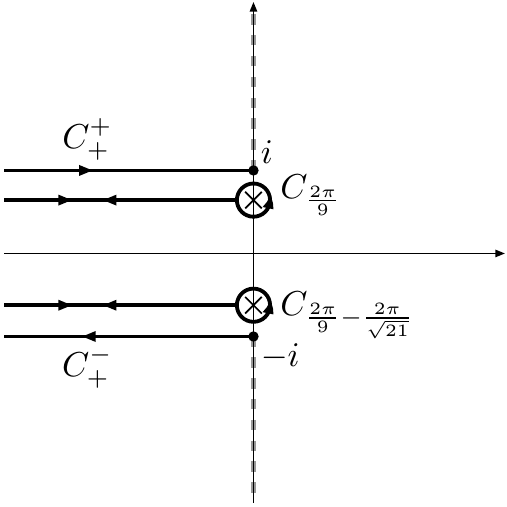}\label{fig:deformed uCplus}}
\caption{Example poles of the integrand and associated deformation of the contour $u(C)$ ($\rho = \frac{1}{\sqrt{21}}$, $\eta = \frac{2\pi}{9}$)}
\label{fig:deformed uC contours}
\end{figure}

Turning to $u(C_+)$, we exchange it for a collection of horizontal and vertical lines together with small loops around the poles of the integrand, displayed in Figure \ref{fig:trans uCplus}.  (Note that the endpoints of $u(C_-)$ and this transitional $u(C_+)$ match.)  The final deformation comes from pulling these vertical components of $u(C_+)$ out to (negative) infinity, allowable due to the exponential decay of the integrand of $S_\alpha(x,\eta)$ for $x > 0$ and $\alpha \eta \not\equiv -\pi$ or $0 \mod{2\pi \rho}$.  We label the keyhole contour surrounding the pole at $u = i \sin(\varphi)$ by $C_\varphi$, and we call the purely horizontal pieces $C_+^\pm$ depending on which branch point $u=\pm i$ they contain.  The resulting contour is shown in Figure \ref{fig:deformed uCplus}.

The contribution to $S_\alpha(x,\eta)$ coming from the $C_\varphi$ pieces of $u(C_+)$ reduces to a residue calculation, for the contributions coming from integration to and from negative infinity along the horizontal components sum to zero.  Noting that $\sigma = +1$ along these contours, a simple calculation shows the residue of the integrand of $S_\alpha(x,\eta)$ at one of these poles is
\begin{multline}
\Res_{u = i \sin(\varphi)} \left\{ \frac{e^{xu}}{4\pi} \cot\!\left[ \frac{\frac{\pi}{2} + \alpha \eta + i \log\!\left[ u + \left(u^2 + 1\right)^\frac{1}{2} \right]}{2\rho} \right] \frac{1}{\left(u^2 + 1\right)^\frac{1}{2}} \right\} \\
= \frac{\rho}{2\pi i} \, \exp\!\left[ i x \sin(\varphi) \right] \; .
\end{multline}
An application of the residue theorem then gives the following.
\begin{lemma}\label{thm:C_varphi contribution}
The contribution to $S_\alpha(x,\eta)$ of one of the pole-enclosing contours $C_\varphi$ is
\begin{multline}\label{eq:C_varphi contribution}
\frac{1}{4\pi} \int_{C_\varphi} e^{xu} \cot\!\left[ \frac{\frac{\pi}{2} + \alpha \eta + i \log\!\left[ u + \sigma \left(u^2 + 1\right)^\frac{1}{2} \right]}{2\rho} \right] \frac{du}{\sigma \left(u^2 + 1\right)^\frac{1}{2}} \\
= \rho \exp\!\left[i x \sin(\varphi) \right] \; .  \qed
\end{multline}
\end{lemma}

We now work with the horizontal pieces $C_\pm^\pm$.  Treating the two contours $C_\pm^\beta$ together, where $\beta = \pm 1$ is the sign of the crossed branch point, we write
\begin{multline}\label{eq:horizontal part of S_alpha 0}
\frac{1}{4\pi} \int_{C_-^\beta \cup C_+^\beta} e^{xu} \cot\!\left[ \frac{\frac{\pi}{2} + \alpha \eta + i \log\!\left[ u + \sigma \left(u^2 + 1\right)^\frac{1}{2} \right]}{2\rho} \right] \frac{du}{\sigma \left(u^2 + 1\right)^\frac{1}{2}} \\
= - \frac{\beta}{4\pi} \int_{-\infty + \beta i}^{\beta i} e^{xu} \cot\!\left[ \frac{\frac{\pi}{2} + \alpha \eta + i \log\!\left[ u - \left(u^2 + 1\right)^\frac{1}{2} \right]}{2\rho} \right] \frac{du}{- \left(u^2 + 1\right)^\frac{1}{2}} \\
\mbox{} - \frac{\beta}{4\pi} \int_{\beta i}^{-\infty + \beta i} e^{xu} \cot\!\left[ \frac{\frac{\pi}{2} + \alpha \eta + i \log\!\left[ u + \left(u^2 + 1\right)^\frac{1}{2} \right]}{2\rho} \right] \frac{du}{\left(u^2 + 1\right)^\frac{1}{2}} \; .
\end{multline}
We make the change of variables
\begin{equation}
s(u) \defeq \sigma(u) \left( \beta i - u \right)^\frac{1}{2}
\end{equation}
along these contours, again taking the principal branch of the square root and correcting with $\sigma$ where necessary.  In the language of Riemann surfaces, this new variable $s$ is a local uniformizer at the branch point $u = \beta i$, which is to say it unravels the doubling action of the map $u(v)$ at that point.  The inverse is given by
\begin{equation}
u(s) = \beta i - s^2 \; .
\end{equation}
Under this change, \eqref{eq:horizontal part of S_alpha 0} becomes
\begin{equation}\label{eq:horizontal part of S_alpha}
\begin{split}
- \frac{\beta}{2\pi} & \int_{-\infty}^0 e^{x\left(\beta i - s^2\right)} \cot\!\left[ \frac{\frac{\pi}{2} + \alpha \eta + i \log\!\left[ \beta i - s^2 - |s|\left(s^2 - 2 \beta i\right)^\frac{1}{2} \right]}{2\rho} \right] \\
&\hspace*{25em} \mbox{} \times \frac{s \, ds}{|s|\left(s^2 - 2\beta i\right)^\frac{1}{2}} \\
&\hspace*{1.5em} \mbox{} + \frac{\beta}{2\pi} \int_{0}^\infty e^{x\left(\beta i - s^2\right)} \cot\!\left[ \frac{\frac{\pi}{2} + \alpha \eta + i \log\!\left[ \beta i - s^2 + |s|\left(s^2 - 2 \beta i\right)^\frac{1}{2} \right]}{2\rho} \right] \\
&\hspace*{25em} \mbox{} \times \frac{s \, ds}{|s|\left(s^2 - 2\beta i\right)^\frac{1}{2}} \; ,
\end{split}
\end{equation}
which we rewrite as
\begin{equation}
\begin{split}
&\frac{\beta \, e^{x \beta i}}{2\pi} \int_{-\infty}^\infty e^{-x s^2} \cot\!\left[ \frac{\frac{\pi}{2} + \alpha \eta + i \log\!\left[ \beta i - s^2 + s\left(s^2 - 2 \beta i\right)^\frac{1}{2} \right]}{2\rho} \right] \\
&\hspace*{25em} \mbox{} \times \frac{ds}{\left(s^2 - 2\beta i\right)^\frac{1}{2}} \; .
\end{split}
\end{equation}
The multiplicative factor of $\beta$ appears here to correct for the direction of integration along the real axis, which varies with the branch point at which we localize; see Figure \ref{fig:contours in s-domain}.

\begin{remark}\label{rmk:method of images}
Pausing for a moment to consider the case where $\rho = \frac{1}{N}$ for $N$ a positive integer, we can see how our analysis reduces to what one expect from the method of images.  Consider the $\alpha = -1$ term in \eqref{eq:horizontal part of S_alpha}.  We can factor out a negative sign in the cotangent to obtain
\begin{multline}
\cot\!\left[ \frac{N}{2} \left( \frac{\pi}{2} - \eta + i \log\!\left[ \beta i - s^2 + s\left(s^2 - 2 \beta i\right)^\frac{1}{2} \right] \right) \right] \\
= - \cot\!\left[ \frac{N}{2} \left( -\frac{\pi}{2} + \eta + i \log\!\left[ -\beta i + s^2 + s\left(s^2 - 2 \beta i\right)^\frac{1}{2} \right] \right) \right] \; .
\end{multline}
Using the fact that cotangent is $\pi$-periodic, we have
\begin{equation}
\begin{split}
&- \cot\!\left[ \frac{N}{2} \left( -\frac{\pi}{2} + \eta + i \log\!\left[ -\beta i + s^2 + s\left(s^2 - 2 \beta i\right)^\frac{1}{2} \right] \right) + N\pi \right] \\
&\hspace*{2em}= - \cot\!\left[ \frac{N}{2} \left( \frac{\pi}{2} + \eta + i \log[-1] + i \log\!\left[ -\beta i + s^2 + s\left(s^2 - 2 \beta i\right)^\frac{1}{2} \right] \right) \right] \\
&\hspace*{2em}= - \cot\!\left[ \frac{N}{2} \left( \frac{\pi}{2} + \eta + i \log\!\left[ \beta i - s^2 - s\left(s^2 - 2 \beta i\right)^\frac{1}{2} \right] \right) \right] \; .
\end{split}
\end{equation}
We now substitute this into \eqref{eq:horizontal part of S_alpha} and make the change of variables $s \leadsto -s$:
\begin{equation}
\begin{split}
&\frac{\beta \, e^{x \beta i}}{2\pi} \int_{-\infty}^\infty e^{-x s^2} \cot\!\left[ \frac{N}{2} \left( \frac{\pi}{2} - \eta + i \log\!\left[ \beta i - s^2 + s\left(s^2 - 2 \beta i\right)^\frac{1}{2} \right] \right) \right] \\
&\hspace*{27em} \mbox{} \times \frac{ds}{\left(s^2 - 2\beta i\right)^\frac{1}{2}} \\
&\hspace*{2em} = - \frac{\beta \, e^{x \beta i}}{2\pi} \int_{-\infty}^\infty e^{-x s^2} \cot\!\left[ \frac{N}{2} \left( \frac{\pi}{2} + \eta + i \log\!\left[ \beta i - s^2 + s\left(s^2 - 2 \beta i\right)^\frac{1}{2} \right] \right) \right] \\
&\hspace*{27em} \mbox{} \times \frac{ds}{\left(s^2 - 2\beta i\right)^\frac{1}{2}} \; .
\end{split}
\end{equation}
Thus, when we sum \eqref{eq:horizontal part of S_alpha} over $\alpha$ and $\beta$, the horizontal parts $C_\pm^\pm$ of the contour sum to zero, and we are left with only the contributions from the $C_\varphi$ contours calculated previously.  Hence, $S(x,\eta)$ is a sum of the terms from Lemma \ref{thm:C_varphi contribution}, and the Schr\"odinger propagagtor is
\begin{equation}
K_{e^{it\Delta}}(r_1,\theta_1,r_2,\theta_2) = -\frac{1}{4\pi i t} \sum_{j=0}^{N-1} \exp\!\left[ \frac{r_1^2 + r_2^2 - 2r_1 r_2 \cos\left( \theta_1 - \theta_2 - \frac{2\pi j}{N} \right)}{4it} \right] \; .
\end{equation}
This is precisely what one obtains from the method of images.  $\diamond$
\end{remark}

\subsection{Preliminary asymptotics}\label{sect:preliminary asymptotics}
At this stage, we can obtain a preliminary asymptotic expansion for the fundamental solution in our regime.  We start with the integral \eqref{eq:horizontal part of S_alpha}, which is amenable to the usual method of saddle points \cite{Olv}.  Its application generates the expansion
\begin{multline}
\frac{\beta \, e^{x \beta i}}{2\pi} \int_{-\infty}^\infty e^{-x s^2} \cot\!\left[ \frac{\frac{\pi}{2} + \alpha \eta + i \log\!\left[ \beta i - s^2 + s\left(s^2 - 2 \beta i\right)^\frac{1}{2} \right]}{2\rho} \right] \frac{ds}{\left(s^2 - 2\beta i\right)^\frac{1}{2}} \\
= \frac{\beta \, e^{\beta \left(x + \frac{\pi}{4}\right) i}}{\left(2\pi\right)^\frac{1}{2}} \, \cot\!\left[ \frac{\left(1 - \beta\right)\frac{\pi}{2} + \alpha \eta}{2\rho} \right] x^{-\frac{1}{2}} + \mathrm{O}\!\left(x^{-\frac{3}{2}}\right) \qquad \text{as $x \To \infty$} \; .
\end{multline}
To obtain an expansion for $S_\alpha(x,\eta)$, we sum over $\beta = \pm 1$ and add to the result the contributions coming from the $C_\varphi$ contours.
\begin{multline}
S_\alpha(x,\eta) = \sum_{\varphi \in \mathcal{P}_\rho(\alpha\eta) \cap \left(-\frac{\pi}{2},\frac{\pi}{2}\right)} \rho \exp[ix \sin(\varphi)] \\
\mbox{} + \left(2\pi x\right)^{-\frac{1}{2}} \left\{ \cot\!\left[\frac{\alpha \eta}{2\rho}\right] e^{\left(x + \frac{\pi}{4}\right)i} - \cot\!\left[\frac{\alpha \eta + \pi}{2 \rho}\right] e^{-\left(x+\frac{\pi}{4}\right)i} \right\} \\
\mbox{} + \mathrm{O}\!\left(x^{-\frac{3}{2}}\right) \qquad \text{as $x \To \infty$} \; .
\end{multline}
Lastly, we sum over $\alpha = \pm 1$, substitute the definitions of $x \defeq \frac{r_1 r_2}{2t}$ and $\eta \defeq \theta_1 - \theta_2$, and multiply by the leading factor $-\frac{1}{4\pi \rho i t} \exp\!\left[\frac{r_1^2+r_2^2}{4it}\right]$ from \eqref{eq:std cone SSK} to obtain the leading order asymptotics of the Schr\"odinger kernel.
\begin{proposition}
Uniformly away from $\theta_1 - \theta_2 \equiv -\pi$, $0$, and $\pi \mod{2\pi \rho}$, the Schr\"odinger propagator has asymptotics
\begin{multline}
K_{e^{it\Delta}}(r_1,\theta_1,r_2,\theta_2) \\
= \frac{1}{t} \left\{ -\frac{1}{4 \pi i} \sum_{-\pi < \pm(\theta_1 - \theta_2) + 2\pi \rho j < 0} \exp\!\left[ \frac{r_1^2 + r_2^2 - 2 r_1 r_2 \, \cos\!\left( \theta_1 - \theta_2 \pm 2\pi\rho j \right)}{4i t} \right] \right. \\
\begin{split}
&\mbox{} + \left( \frac{r_1 r_2}{2t} \right)^{-\frac{1}{2}} \left( \frac{i}{16 \pi^3 \rho^2} \right)^\frac{1}{2} \exp\!\left[ \frac{\left(r_1 + r_2\right)^2}{4 i t} \right] \left\{ \cot\!\left[ \frac{\theta_1 - \theta_2 + \pi}{2\rho} \right]  \right. \\
&\hspace*{22em} \left. \mbox{} - \cot\!\left[ \frac{ \theta_1 - \theta_2 - \pi}{2\rho} \right] \right\} 
\end{split}\\
\left. \mbox{} +  \mathrm{O}\!\left( \left( \frac{r_1 r_2}{2t} \right)^{-\frac{3}{2}}\right) \right\} \qquad \text{as $\dfrac{r_1 r_2}{2t} \To \infty$} \; .
\end{multline}
\end{proposition}

Note that the coefficients in this expansion diverge as a pole moves toward one of the branch points, i.e.\ as $\theta_1 - \theta_2$ approaches $-\pi$, $0$, or $\pi \mod{2\pi\rho}$.  In particular, the singularity at $\theta_1 - \theta_2 = 0$ remains because the asymptotics of $S_\alpha(x,\eta)$ are invalid when $\eta = 0$.  These singularities arise due to the phase of the integrals \eqref{eq:horizontal part of S_alpha} being independent of $\eta = \theta_1 - \theta_2$, unlike the location of the poles.  To develop uniform asymptotics as we approach this interface, we must treat the part of the integrand causing this divergence separately.  Thus, we return to the integral \eqref{eq:horizontal part of S_alpha}.

\subsection{Uniform asymptotics approaching the interface}\label{sect:away from geo/diff interface}
Let $A_{\alpha,\beta}(s)$ be the amplitude of this integral \eqref{eq:horizontal part of S_alpha}, i.e.\
\begin{equation}
A_{\alpha,\beta}(s) \defeq \frac{\beta \, e^{x \beta i}}{2\pi} \cot\!\left[ \frac{\frac{\pi}{2} + \alpha \eta + i \log\!\left[ \beta i - s^2 + s\left(s^2 - 2 \beta i\right)^\frac{1}{2} \right]}{2\rho} \right] \frac{1}{\left(s^2 - 2\beta i\right)^\frac{1}{2}} \; .
\end{equation}
Its poles are located at $s = \sigma_\varphi \left( \beta i - i \sin(\varphi) \right)^\frac{1}{2}$, where $\sigma_\varphi \defeq \sgn(\cos(\varphi))$ distinguishes from which $v_\pm$-sheet the pole originally comes and $\varphi$ is an element of $\mathcal{P}_\rho(\alpha\eta) = \left\{ \frac{\pi}{2} + \alpha \eta + 2\pi\rho k ;\; k \in \mathbb{Z} \right\}$.
\begin{figure}
	\subfigure[The $s$-plane for $\beta = -1$]{\includegraphics{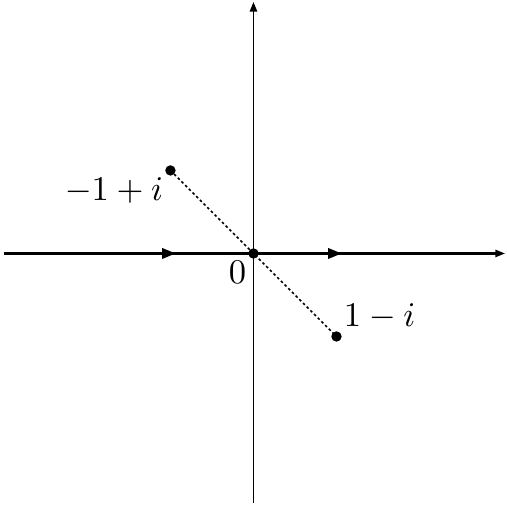}}
	\subfigure[The $s$-plane for $\beta = 1$]{\includegraphics{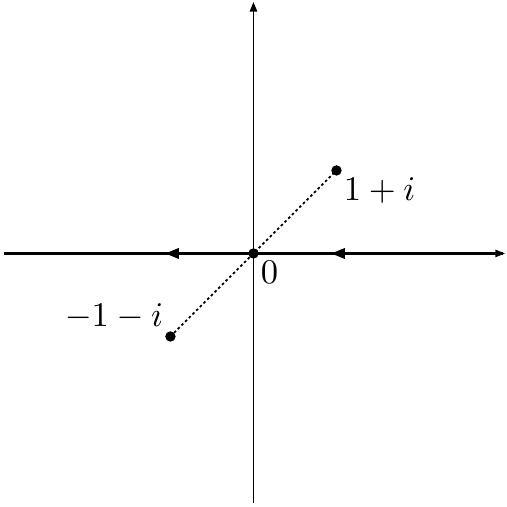}}
\caption{The $s$-planes and location of poles of the integrand}
\label{fig:contours in s-domain}
\end{figure}
We can therefore write $A_{\alpha,\beta}(s)$ as the sum
\begin{equation}\label{eq:defn of B_alpha,beta}
A_{\alpha,\beta}(s) = \sum_{\varphi \in \mathcal{P}_\rho(\alpha\eta)} \frac{a^{(\alpha,\beta)}_\varphi}{s - \sigma_\varphi \left( \beta i - i \sin(\varphi) \right)^\frac{1}{2}} + B_{\alpha,\beta}(s) \; ,
\end{equation}
where $a_\varphi^{(\alpha,\beta)}$ is the residue of $A_{\alpha,\beta}(s)$ at the pole $s = \sigma_\varphi \left(\beta i - i \sin(\varphi) \right)^\frac{1}{2}$ and $B_{\alpha,\beta}(s)$ is holomorphic in $s$ away from $s = \pm \sqrt{2} \, e^{\beta \frac{\pi}{4} i}$.  A now familiar calculation determines the residues $a_\varphi^{(\alpha,\beta)}$ to be 
\begin{equation}
a_\varphi^{(\alpha,\beta)} = -\frac{\beta \rho \, e^{x \beta i}}{2 \pi i} \; .
\end{equation}
Hence, developing the asymptotics of $\int_{-\infty}^\infty e^{-xs^2} \, A_{\alpha,\beta}(s) \, ds$ is the same as developing those of 
\begin{equation}\label{eq:sep horizontal part of S_alpha}
\sum_{\varphi \in \mathcal{P}_\rho(\alpha\eta)} a_\varphi^{(\alpha,\beta)} \int_{-\infty}^\infty \frac{e^{-x s^2}}{s - \sigma_\varphi \left(\beta i - i \sin(\varphi) \right)^\frac{1}{2}} \, ds + \int_{-\infty}^\infty e^{-x s^2} \, B_{\alpha,\beta}(s) \, ds \; .
\end{equation}

\begin{lemma}\label{thm:pole contribution C_pm^pm}
The contribution to $\int_{-\infty}^\infty e^{-x s^2} \, A_{\alpha,\beta}(s) \, ds$ from the poles of $A_{\alpha,\beta}(s)$ is
\begin{multline}\label{eq:pole contribution C_pm^pm}
\sum_{\varphi \in \mathcal{P}_\rho(\alpha\eta)} a_\varphi^{(\alpha,\beta)} \int_{-\infty}^\infty \frac{e^{-x s^2}}{s - \sigma_\varphi \left(\beta i - i \sin(\varphi) \right)^\frac{1}{2}} \, ds \\
= - \frac{\rho}{2} \sum_{\varphi \in \mathcal{P}_\rho(\alpha\eta)} \sigma_\varphi \exp\!\left[ i x \sin(\varphi) \right] \erfc\!\left[ e^{-\beta \frac{\pi}{4}i} \, x^\frac{1}{2} \left(1 - \beta \sin(\varphi) \right)^\frac{1}{2} \right] \; ,
\end{multline}
where $\erfc[z]$ is the complementary error function,
\begin{equation}
\erfc[z] \defeq \frac{2}{\sqrt{\pi}} \int_z^\infty e^{-t^2} \, dt \; .
\end{equation}
\end{lemma}

\begin{proof}
We begin with calculating the integral in \eqref{eq:pole contribution C_pm^pm} via the formula \cite{AbrSte}*{(7.1.4)}
\begin{equation}
\int_{-\infty}^\infty \frac{e^{-t^2}}{t - a} \, dt = i \pi \, e^{-a^2} \, \erfc[-ia] \qquad \text{if $\Im[a] > 0$} \; .
\end{equation}
After the change of variables $t = x^\frac{1}{2} s$, this equality becomes
\begin{equation}
\int_{-\infty}^\infty \frac{e^{-x s^2}}{s - a} \, ds = i \pi \, e^{-x a^2} \, \erfc\!\left[-i x^\frac{1}{2} a\right] \qquad \text{if $\Im[a] > 0$} \; ,
\end{equation}
and making the change $s \leadsto - s$ gives us a formula for $\Im[a] < 0$:
\begin{equation}
\int_{-\infty}^\infty \frac{e^{-x s^2}}{s - a} \, ds = - i \pi \, e^{-x a^2} \, \erfc\!\left[i x^\frac{1}{2} a\right] \qquad \text{if $\Im[a] < 0$} \; .
\end{equation}
Noting that
\begin{equation}
\sgn\!\left[ \Im\!\left[ \sigma_\varphi \left(\beta i - i \sin(\varphi) \right)^\frac{1}{2} \right] \right] = \sgn\!\left[ \Im\!\left[ \sigma_\varphi \, e^{\beta \frac{\pi}{4} i} \left(1 - \beta \sin(\varphi) \right)^\frac{1}{2} \right] \right] = \sigma_\varphi \beta
\end{equation}
and applying the above formulae shows
\begin{multline}
a_\varphi^{(\alpha,\beta)} \int_{-\infty}^\infty \frac{e^{-x s^2}}{s - \sigma_\varphi \left(\beta i - i \sin(\varphi) \right)^\frac{1}{2}} \, ds \\
= - \frac{\sigma_\varphi \rho}{2} \exp\!\left[ i x \sin(\varphi) \right] \erfc\!\left[ e^{-\beta \frac{\pi}{4}i} \, x^\frac{1}{2} \left(1 - \beta \sin(\varphi) \right)^\frac{1}{2} \right] \; .
\end{multline}
Summing over $\varphi$ in $\mathcal{P}_\rho(\alpha\eta)$ proves the lemma.
\end{proof}

Recall that the complementary error function $\erfc[z]$ is entire with everywhere convergent Taylor series
\begin{equation}
\erfc[z] = 1 - \frac{2 \, e^{-z^2}}{\sqrt{\pi}} \sum_{k = 0}^\infty \frac{2^k \, z^{2k + 1}}{(2k+1)!!} \; ,
\end{equation}
and it has the asymptotic expansion 
\begin{equation}
\erfc[z] \sim \frac{e^{-z^2}}{\sqrt{\pi} z} \sum_{k=0}^\infty \frac{(-1)^k \, (2k)!}{k! \, (2z)^{2k}} \qquad \text{as $z \To \infty$ in $\left|\arg(z)\right| < \dfrac{3\pi}{4}$} \; .
\end{equation}
Thus the terms in \eqref{eq:pole contribution C_pm^pm} are $\mathrm{O}\!\left( x^{-\frac{1}{2}} \right)$ as $x \To \infty$ with $\eta$ fixed.  In particular, they are uniformly bounded for all $x$ and $\eta$ in the current regime.

Returning to the calculation, we are left with the asymptotic development of the remainder term in \eqref{eq:sep horizontal part of S_alpha}.  This is the content of the next lemma.
\begin{lemma}\label{thm:B_alpha,beta asymptotics}
The integral $\int_{-\infty}^\infty e^{-x s^2} \, B_{\alpha,\beta}(s) \, ds$ has asymptotic expansion
\begin{equation}
\int_{-\infty}^\infty e^{-x s^2} \, B_{\alpha,\beta}(s) \, ds \sim \sum_{k=0}^\infty b_{2k}^{(\alpha,\beta)} \, \Gamma\!\left( \frac{2k+1}{2} \right) \, x^{- \frac{2k+1}{2}} \qquad \text{as $x \To \infty$} \; ,
\end{equation}
where the $b_k^{(\alpha,\beta)}$ are the Taylor coefficients of $B_{\alpha,\beta}(s)$ at $s = 0$.  In particular, the leading coefficient $b_0^{(\alpha,\beta)} = b_0^{(\alpha,\beta)}(x,\eta)$ is
\begin{multline}
b_0^{(\alpha,\beta)}(x,\eta) = \frac{\beta \, e^{\beta \left(x - \frac{\pi}{4}\right)i}}{2\pi i} \left\{ - \frac{\beta}{\sqrt{2}} \, \cot\!\left[ \frac{(1 - \beta) \, \frac{\pi}{2} + \alpha \eta}{2 \rho} \right] \phantom{\sum_{\varphi \in \mathcal{P}_\rho(\alpha\eta)}}\right. \\
\left. \mbox{} - \rho \sum_{\varphi \in \mathcal{P}_\rho(\alpha\eta)} \sigma_\varphi \, \left(1 - \beta \sin(\varphi)\right)^{-\frac{1}{2}} \right\} \; .
\end{multline}
\end{lemma}

\begin{proof}
We first note that $A_{\alpha,\beta}(s)$ can be written in the form
\begin{equation}
A_{\alpha,\beta}(s) = - \frac{\beta \, e^{x \beta i}}{2 \pi i \left(s^2 - 2 \beta i \right)^\frac{1}{2}} \, \frac{\exp\!\left[ i \, \frac{\frac{\pi}{2} + \alpha\eta}{\rho} \right] + \left[ \beta i - s^2 + s \left( s^2 - 2 \beta i \right)^\frac{1}{2} \right]^\frac{1}{\rho}}{\exp\!\left[ i \, \frac{\frac{\pi}{2} + \alpha\eta}{\rho} \right] - \left[ \beta i - s^2 + s \left( s^2 - 2 \beta i \right)^\frac{1}{2} \right]^\frac{1}{\rho}} \; , 
\end{equation}
where the roots have branch $[-\pi,\pi)$.  This shows that $A_{\alpha,\beta}(s)$ is a ratio of functions which are holomorphic away from branch points at $s = \pm \sqrt{2} e^{\beta \frac{\pi}{4} i}$ and poles at $s = \sigma_\varphi \left(\beta i - i \sin(\varphi \right)^\frac{1}{2}$; in particular, the $\rho$th root does not effect the holomorphy of $A_{\alpha,\beta}(s)$ since its argument is never zero for finite $s$.  Thus, when we remove the poles to form $B_{\alpha,\beta}(s)$, we are left with a function which is holomorphic away from these branch points at $s = \pm \sqrt{2} \, e^{\beta \frac{\pi}{4} i}$ and bounded for $s$ real.

Moving on the the proof of the statement, we write $B_{\alpha,\beta}(s)$ as a Taylor series with remainder,
\begin{equation}
B_{\alpha,\beta}(s) = \sum_{k = 0}^{2N-1} b_k^{(\alpha,\beta)} \, s^k + R_{2N}^{(\alpha,\beta)}(s) \; .
\end{equation}
Integration of both sides along the real line produces
\begin{equation}\label{eq:B_alpha,beta expansion}
\begin{aligned}
\int_{-\infty}^\infty e^{-x s^2} \, B_{\alpha,\beta}(s) \, ds &= \sum_{k = 0}^{2N-1} b_k^{(\alpha,\beta)} \int_{-\infty}^\infty e^{-x s^2} \, s^k \, ds + \int_{-\infty}^\infty e^{-x s^2} \, R_{2N}^{(\alpha,\beta)}(s) \, ds \\
&= \sum_{k = 0}^{N-1} b_{2k}^{(\alpha,\beta)} \, \Gamma\!\left(\frac{2k+1}{2}\right) \, x^{-\frac{2k+1}{2}} + \int_{-\infty}^\infty e^{-x s^2} \, R_{2N}^{(\alpha,\beta)}(s) \, ds \; .
\end{aligned}
\end{equation}
This leaves us to gauge the size of the remainder term.  We note that $R_{2N}^{(\alpha,\beta)}(s)$ is $\mathrm{O}\!\left(|s|^{2N-1}\right)$ since $B_{\alpha,\beta}(s)$ is holomorphic away from $s = \pm \sqrt{2} e^{\beta \frac{\pi}{4} i}$ and bounded for $s$ real, and this implies
\begin{equation}\label{eq:R^alpha,beta_N bound}
\int_{-\infty}^\infty e^{-xs^2} \, R_{2N}^{(\alpha,\beta)}(s) \, ds = \mathrm{O}\left( x^{-(N+1)} \right) \; .
\end{equation}
This proves the first statement of the lemma.  The computation of $b_0^{(\alpha,\beta)}(\eta)$ is left to the reader; it follows immediately from the definition \eqref{eq:defn of B_alpha,beta} via some simple algebraic manipulation.
\end{proof}

Combining the results of Lemma \ref{thm:pole contribution C_pm^pm} and Lemma \ref{thm:B_alpha,beta asymptotics}, we acquire an asymptotic expansion  for the integral $\int_{-\infty}^\infty e^{-x s^2} \, A_{\alpha,\beta}(s) \, ds$ in decreasing powers of $x$ as $x \To \infty$.  Unlike the previous asymptotics, this expansion remains valid as $\alpha\eta$ approaches $-\pi$ or $0 \mod{2\pi\rho}$.  Namely,
\begin{multline}\label{eq:A_alpha,beta asymptotics}
\int_{-\infty}^\infty e^{-x s^2} \, A_{\alpha,\beta}(s) \, ds \\
\sim - \frac{\rho}{2} \sum_{\varphi \in \mathcal{P}_\rho(\alpha\eta)} \sigma_\varphi \exp\!\left[ i x \sin(\varphi) \right] \erfc\!\left[ e^{-\beta \frac{\pi}{4}i} \, x^\frac{1}{2} \left(1 - \beta \sin(\varphi) \right)^\frac{1}{2} \right] \\
\mbox{} + \sum_{k=0}^\infty \Gamma\!\left(\frac{2k+1}{2}\right) b_{2k}^{(\alpha,\beta)}(x,\eta) \, x^{-\frac{2k+1}{2}} \qquad \text{as $x \To \infty$} \; .
\end{multline}
We develop the asymptotics of $S_\alpha(x,\eta)$ in this regime by summing \eqref{eq:A_alpha,beta asymptotics} over $\beta$ and including the contributions of the $C_\varphi$ contours from Lemma \ref{thm:C_varphi contribution}, giving the expansion
\begin{multline}\label{eq:S_alpha(x,eta) asymptotics}
S_\alpha(x,\eta) \sim \rho \sum_{\varphi \in \mathcal{P}_\rho(\alpha\eta)} \exp[ix \sin(\varphi)] \\
\begin{split}
&\mbox{} \times \left\{ \frac{1 + \sigma_\varphi}{2} - \frac{\sigma_\varphi}{2} \erfc\!\left[ e^{-\frac{\pi}{4}i} \, x^\frac{1}{2} \left(1 - \sin(\varphi) \right)^\frac{1}{2} \right] \right. \\
&\hspace*{10em} \left. \mbox{} - \frac{\sigma_\varphi}{2} \erfc\!\left[ e^{\frac{\pi}{4}i} \, x^\frac{1}{2} \left(1 + \sin(\varphi) \right)^\frac{1}{2} \right] \right\} 
\end{split}\\
\mbox{} + \sum_{k=0}^\infty \Gamma\!\left(\frac{2k+1}{2}\right) \left\{ b_{2k}^{(\alpha,+)}(x,\eta) + b_{2k}^{(\alpha,-)}(x,\eta) \right\} x^{-\frac{2k+1}{2}} \qquad \text{as $x \To \infty$} \; .
\end{multline}

The final steps in our calculation of the asymptotics of $K_{e^{it\Delta}}$ in the non-interactive regime is to sum over $\alpha = \pm 1$ and to substitute the resulting asymptotics for $S(x,\eta)$ into the expression \eqref{eq:std cone SSK}, converting from our dummy variables in the process.  In the interest of clarity, though, we will delay the summing over $\alpha$ and first introduce some notation for the terms appearing in the expansions obtained by substituting the $S_\alpha(x,\eta)$ asymptotics into \eqref{eq:std cone SSK}.  The first of these is a more explicit version of $\sigma_\varphi$,
\begin{equation}
\Sigma^\alpha_j(\theta_1,\theta_2) \defeq \sgn\!\left( \cos\!\left( \frac{\pi}{2} + \alpha (\theta_1 - \theta_2) + 2\pi\rho j \right) \right) \; .
\end{equation}
We also introduce a name for the terms one would obtain purely from the formal application of the method of images,
\begin{multline}
\tilde{\mathsf{G}}_0^\alpha(j;t,r_1,\theta_1,r_2,\theta_2) \\
\defeq - \frac{1}{4\pi i} \, \exp\!\left[ \frac{r_1^2 + r_2^2 - 2r_1 r_2 \cos(\theta_1 - \theta_2 + \alpha \, 2\pi\rho j)}{4it} \right] \; ,
\end{multline}
and we let $\mathsf{G}_0^\alpha(t,r_1,\theta_1,r_2,\theta_2)$ be their sum
\begin{equation}
\mathsf{G}_0^\alpha(t,r_1,\theta_1,r_2,\theta_2) \defeq \sum_{-\pi < \alpha(\theta_1 - \theta_2) + 2\pi\rho j < 0} \tilde{\mathsf{G}}_0^\alpha(j) \; .
\end{equation}
The indices of summation here come from the set of pole phases, which we recall is
\begin{equation}
\mathcal{P}_\rho\!\left(\alpha(\theta_1 - \theta_2) \right) \defeq \left\{ \frac{\pi}{2} + \alpha(\theta_1 - \theta_2) + 2\pi\rho j ;\; j \in \mathbb{Z} \right\} \cap [-\pi,\pi) \; ,
\end{equation}
intersected with the interval $\left(-\frac{\pi}{2},\frac{\pi}{2}\right)$.  This restricts the phases to those whose corresponding poles are surrounded by the $C_\varphi$ contours.  We next define a function representing the sum of complementary error functions appearing in the expansion \eqref{eq:S_alpha(x,eta) asymptotics},
\begin{multline}
\tilde{\mathsf{G}}^\alpha_{-\frac{1}{2}}(j;t,r_1,\theta_1,r_2,\theta_2) \\
\defeq \frac{\Sigma^\alpha(j;\theta_1,\theta_2)}{8\pi i} \left( \frac{r_1 r_2}{2t} \right)^\frac{1}{2} \exp\!\left[ \frac{r_1^2 + r_2^2 - 2r_1 r_2 \cos(\theta_1 - \theta_2 + \alpha \, 2\pi\rho j)}{4it} \right] \\
\mbox{} \times \left\{\erfc\!\left[  \left(-\frac{ir_1 r_2}{2t}\right)^\frac{1}{2} \left(1 - \cos(\theta_1 - \theta_2 +\alpha \, 2\pi \rho j) \right)^\frac{1}{2} \right] \right. \\
\left. \mbox{} + \erfc\!\left[ \left(\frac{i r_1 r_2}{2t}\right)^\frac{1}{2} \left(1 + \cos(\theta_1 - \theta_2 +\alpha \, 2\pi \rho j) \right)^\frac{1}{2} \right] \right\} \; ,
\end{multline}
and similarly their sum is denoted by
\begin{equation}
\mathsf{G}_{-\frac{1}{2}}^\alpha(t,r_1,\theta_1,r_2,\theta_2) \defeq \sum_{-\frac{3\pi}{2} \leqslant \alpha(\theta_1 - \theta_2) + 2\pi\rho j < \frac{\pi}{2}} \tilde{\mathsf{G}}_{-\frac{1}{2}}^\alpha(j;t,r_1,\theta_1,r_2,\theta_2) \; .
\end{equation}
This sum ranges over all of the corresponding phases in $\mathcal{P}_\rho\!\left(\alpha(\theta_1 - \theta_2)\right)$, producing the wider range of possible indices $j$.  Note that we have inserted a factor of $\left( \frac{r_1 r_2}{2t} \right)^\frac{1}{2}$ in the definition of $\tilde{\mathsf{G}}_{-\frac{1}{2}}^\alpha$; this is just a psychological convenience.  It makes $\tilde{\mathsf{G}}^\alpha_{-\frac{1}{2}}$ into a $\mathrm{O}(1)$ function in $\frac{r_1r_2}{2t}$, which allows us to write the decay in $\frac{r_1 r_2}{2t}$ of each term explicitly in the asymptotic expansion.  Lastly, we introduce
\begin{multline}
\mathsf{D}^\alpha_{-\frac{2k+1}{2}}(t,r_1,\theta_1,r_2,\theta_2) \defeq - \frac{\Gamma\!\left(\frac{2k+1}{2}\right)}{8 \pi^2 \rho t \cdot (2k)!} \sum_{\beta = \pm 1} \exp\!\left[ \frac{(r_1 - \beta r_2)^2}{4it} \right] \\
\mbox{} \times \frac{d^{2k}}{ds^{2k}} \left\{ \cot\!\left[ \frac{\frac{\pi}{2} + \alpha(\theta_1 - \theta_2) + i \log\!\left[\beta i - s^2 + s \left(s^2 - 2 \beta i \right)^\frac{1}{2} \right]}{2\rho} \right] \right. \\
\left. \left. \mbox{} - \sum_{- \frac{3\pi}{2} \leqslant \alpha (\theta_1 - \theta_2) + 2\pi \rho j < \frac{\pi}{2}} \frac{i \rho}{s - \Sigma^\alpha_j(\theta_1,\theta_2) \left(\beta i - i \cos(\theta_1 - \theta_2 + \alpha \, 2\pi \rho j \right)^\frac{1}{2}} \right\} \right\vert_{s=0}
\end{multline}
to represent the coefficients of $x^{-\frac{2k+1}{2}}$ in the asymptotic expansions of the integrals $\int_{-\infty}^\infty e^{-x s^2} B_{\alpha,\beta}(s) \, ds$.  
%

While the functions which arise from the poles, that is $\mathsf{G}^\alpha_0(t,r_1,\theta_1,r_2,\theta_2)$ and $\mathsf{G}^\alpha_{-\frac{1}{2}}(t,r_1,\theta_1,r_2,\theta_2)$, are easily seen to be uniformly bounded for $\alpha\eta \not\equiv -\pi$ or $0 \mod{2\pi\rho}$, we emphasize that the terms $\mathsf{D}^\alpha_{-\frac{2k+1}{2}}(t,r_1,\theta_1,r_2,\theta_2)$ arising from the remainder terms are bounded in the same regime by our construction.  That is, inspection shows they uniformly bounded (and smooth) in $t$, $r_1$, and $r_2$, and since we have removed the poles arising from change in $\theta_1$ and $\theta_2$, there are no singularities in these variables.  The piecewise-smoothness of all these functions then follows because they were smooth away from the poles, but jumps that occur when poles join or leave the set of pole phases or cross a branch point remain.

To conclude this section, we state the asymptotics of $K_{e^{it\Delta}}$ in the following theorem, suppressing the dependence of the above functions on the variables $t$, $(r_1,\theta_1)$, and $(r_2,\theta_2)$.
\begin{theorem}\label{thm:propagator asymptotics away from interface}
For $\theta_1 - \theta_2 \not\equiv -\pi$, $0$, or $\pi \mod{2\pi \rho}$, the Schwartz kernel of $e^{it\Delta}$ has the asymptotic expansion
\begin{multline}\label{eq:asymptotics away from interface}
K_{e^{it\Delta}}(r_1,\theta_1,r_2,\theta_2) \\
\mbox{} \sim \frac{1}{t} \sum_{\alpha = \pm 1} \left\{ \mathsf{G}_0^\alpha + \mathsf{G}^\alpha_{-\frac{1}{2}} \left(\frac{r_1 r_2}{2t} \right)^{-\frac{1}{2}}  + \sum_{k=0}^\infty \mathsf{D}^\alpha_{-\frac{2k+1}{2}} \left(\frac{r_1 r_2}{2t}\right)^{-\frac{2k + 1}{2}} \right\}\\
\text{as $\dfrac{r_1 r_2}{2t} \To \infty$} \; . \qed
\end{multline}
\end{theorem}

\section{Strichartz estimates}\label{sect:Strichartz}

We will now apply the information gained from the asymptotic development of $K_{e^{it\Delta}}$ in Section \ref{sect:Asymptotics} to prove the Strichartz estimates for the solution operator
\begin{equation}
\mathcal{U}(t) f(r,\theta) \defeq \int_{\theta'=0}^{2\pi \rho} \int_{r' = 0}^\infty K_{e^{it\Delta}}(r,\theta,r',\theta') \, f(r',\theta') \, r' dr' d\theta'
\end{equation}
of the Schr\"odinger equation on $C(\mathbb{S}^1_\rho)$.  Thus we end with the following theorem.  

\begin{theorem}
Suppose $\frac{2}{p} + \frac{2}{q} = 1$ and $\frac{2}{\tilde{p}} + \frac{2}{\tilde{q}} = 1$.  Then the Schr\"odinger solution operator $\mathcal{U}(t) \defeq e^{it\Delta}$ on $C(\mathbb{S}^1_\rho)$ satisfies the Strichartz estimates
\begin{equation}\label{eq:Strichartz 1}
\left\| \mathcal{U}(t) f(r,\theta) \right\|_{L^p_t L^q(r \, dr d\theta)} \lesssim \left\|f\right\|_{L^2(r \, dr d\theta)}
\end{equation}
\begin{equation}\label{eq:Strichartz 2}
\left\| \int \mathcal{U}(-s) F(s,r,\theta) \, ds \right\|_{L^2(r \, dr d\theta)} \lesssim \left\| F(t,r,\theta) \right\|_{L^{p'}_t L^{q'}(r \, dr d\theta)}
\end{equation}
\begin{equation}\label{eq:Strichartz 3}
\left\| \int_{s < t} \mathcal{U}(t-s) F(s,r,\theta) \, ds \right\|_{L^p_t L^q(r \, dr d\theta)} \lesssim \left\| F(t,r,\theta) \right\|_{L^{\tilde{p}}_t L^{\tilde{q}}(r \, dr d\theta)} \; .
\end{equation}
\end{theorem}

\begin{proof}
To prove the estimates, we will utilize the abstract Strichartz estimate of Keel and Tao \cite{KeeTao}.  Their theorem states that the estimates \eqref{eq:Strichartz 1}, \eqref{eq:Strichartz 2}, and \eqref{eq:Strichartz 3} are implied by $L^2$-boundedness,
\begin{equation}
\left\| \mathcal{U}(t) f(r,\theta) \right\|_{L^2(r \, dr d\theta)} \lesssim \left\| f \right\|_{L^2(r \, dr d\theta)} \; ,
\end{equation}
and a dispersive estimate\footnote{Ignoring the cone tip, one could heuristically expect such a dispersive estimate to hold on $C(\mathbb{S}^1_\rho)$ by appealing to the calculation in \cite{HasWun} and utilizing the fact that these cones have no conjugate points.},
\begin{equation}
\left\|\mathcal{U}(t-s) g(r,\theta) \right\|_{L^\infty} \lesssim \left| t - s \right|^{-1} \left\| g \right\|_{L^1(r \, dr d\theta)} \; .
\end{equation}
Here, $f$ ranges over $L^2\!\left(C(\mathbb{S}^1_\rho)\right)$ and $g$ over $L^1\!\left(C(\mathbb{S}^1_\rho)\right)$.  The first estimate follows from unitarity of $\mathcal{U}(t)$ on $L^2\!\left(C(\mathbb{S}_\rho^1)\right)$.  The second is implied by the claim that $K_{e^{it\Delta}}$ is an element of $t^{-1} \, L^\infty\!\left( \mathbb{R} \times C(\mathbb{S}^1_\rho) \right)$.

Noting that the claim is implied by an $L^\infty$ bound on $S(x,\eta)$, we consider separately the cases where $x \leqslant 1$ and $x \geqslant 1$.  In the former, the claim follows from the computations in the proof of Proposition \ref{thm:propagator asymptotics near cone tip}.  For the case $x \geqslant 1$, it is implied by the calculations \eqref{eq:C_varphi contribution} and \eqref{eq:pole contribution C_pm^pm} of the pole contributions; the expansion \eqref{eq:B_alpha,beta expansion} of the integral over the horizontal contours; the bound \eqref{eq:R^alpha,beta_N bound} for the remainder in this expansion; and the fact that the interface of the geometric and diffractive fronts, where these calculations do not hold, is measure zero in $C(\mathbb{S}^1_\rho)$.  This concludes the proof.
\end{proof}


\end{document}